\documentclass[12pt]{article}
\usepackage{amsfonts,amsmath,amsthm, amssymb}
\usepackage{latexsym, euscript, epic, eepic}
\newtheorem{theorem}{Theorem}[section]
\newtheorem{definition}{Definition}[section]
\newtheorem{proposition}{Proposition}[section]

\newtheorem{remark}{Remark}
\newtheorem{example}{Example}
\newtheorem{maintheorem}{Theorem}

\def \a{\alpha}

\begin{document}

\title{Conditional Ergodic Averages for Asymptotically  Additive
Potentials
 \footnotetext {2010 Mathematics Subject Classification: 37A05, 37C45, 37A30}}
\author{   Yun Zhao \\
\small \it  Department of Mathematics, Soochow University, Suzhou 215006, Jiangsu, P.R.China\\
\small \it e-mail: zhaoyun@suda.edu.cn}
\date{}
 \maketitle

 \begin{center}
\begin{minipage}{120mm}
{\small {\bf Abstract.} Using an asymptotically additive sequence
of continuous functions as a restrictive condition, this paper
studies the relations of several ergodic averages for
asymptotically additive potentials. Basic properties of
conditional maximum ergodic averages are studied. In particular,
if the dynamical systems satisfy the specification property,  the maximal growth rate of an asymptotically
additive potential on the level set is equal to its conditional
maximum ergodic averages and the maximal growth rates on the
irregular set is its maximum ergodic averages. Finally, the
applications for suspension flows are given in the end of the
paper.}
\end{minipage}
\end{center}

\vskip0.5cm

{\small{\bf Key words and phrases} \ Ergodic optimization;
Multifractal analysis; Asymptotically additive potentials
}\vskip0.5cm

 \section{Introduction}
  Throughout this paper $X$ is a compact metric space with metric $d$ and $T:X\rightarrow X$ is a continuous transformation. Such a tuple  $(X,T)$ is called a \emph{topological dynamical systems} (TDS for short).  Let $C(X)$ denote the Banach space of continuous functions from $X$ to
$\mathbb{R}$ with supremum norm $\|\cdot\|$. Let
$\mathcal{M}_T$ and $\mathcal{E}_T$ denote  the space of
$T$-invariant Borel probability measures on $X$ and the set of all
$T$-invariant ergodic Borel probability measures on $X$, respectively.

Given a continuous function $f:X\rightarrow \mathbb{R}$, the
\emph{maximum ergodic average} for $f$ is defined as
\[
\beta (f):=\sup_{\mu\in \mathcal{M}_T}\int f \mathrm{d}\mu.
\]
 Since $\mathcal{M}_T$
is weak$^*$ compact and the map $\mu\mapsto \int f \mathrm{d}\mu$
is continuous, there always exists a measure which attains the
supremum in the above formula, such a measure is called
\emph{$f-$maximizing measure}. Jenkinson \cite{je} proved the
following basic relations between different time averages of a
continuous function  $f$
\[
\beta (f)=\sup_{x\in X} \limsup_{n\rightarrow\infty} \frac 1 n S_n
f(x)=\lim_{n\rightarrow\infty}\frac 1 n \max_{x\in X}S_n
f(x)=\sup_{x\in \mathrm{Reg}(f,T)}\lim_{n\rightarrow\infty}\frac 1
n S_n f(x)
\]
where $S_n f(x):=\sum_{i=0}^{n-1}f(T^ix)$ and $\mathrm{Reg}(f,T)$
is the set of points $x\in X$ such that the limit of the sequence
$\{\frac 1 n S_n f(x)\}_{n\ge1}$ exists. Let
\[
\mathcal{M}_{\max}(f) := \left \{\mu \in  \mathcal{M}_T :\ \int f
\mathrm{d}\mu=\beta (f)\right\}
\]
denote the set of all $f-$maximizing measures. The study of the
variational problem of the functional $\beta(\cdot)$ and the set
$\mathcal{M}_{\max}(\cdot)$ has been termed {\it ergodic
optimization}, and has attracted some recent research interest
\cite{bo1,bo2,bj,br,bq,clt,je,mo1,mo2,mo3}.  Analogous problems for sub-additive
potentials has been investigated  for deterministic
dynamical systems in \cite{dai11,sch,ss00} and  for random dynamical systems in \cite{cao}.
In \cite{zhao}, author  studied  the sub-growth rate of
asymptotically sub-additive potentials and subordination principle
for sub-additive potentials. One of the motivation of the present paper
is to study the  ergodic optimization for a particular sequence of
\emph{asymptotically additive potentials} (see precise definition
in the next section and use AAP for short) which arises naturally
in the study of the dimension theory in dynamical systems (see
\cite{fh,zzc} for examples of AAP).

The study of this paper is also motivated by the theory of
multifractal analysis. The theory of multifractal analysis is a
subfield of the dimension theory of dynamical systems, its main
purpose is to study the complexity of the level sets or irregular
sets of invariant local quantities obtained from a given dynamical
system, e.g., the topological entropy or pressure on these sets.
See the books \cite{ba3,pe} for details about the theory of
multifractal analysis.  For a sequence of asymptotically additive
continuous functions $\Phi=\{\varphi_n\}_{n\ge 1}$, the \emph{level
sets} induced by the asymptotically additive potential $\Phi$ are
defined by
\[
K_\Phi(\alpha): =\{ x\in X\ :\ \lim_{n\rightarrow\infty}
\frac{1}{n}\varphi_n (x)=\alpha \}.
\]
Since these level sets are pairwise disjoint for different real
numbers $\alpha$,  they induce the natural decomposition
\[
X=\widehat{X}_\Phi\cup\bigcup_{\a\in \mathbb{R}} K_\Phi(\alpha)
\]
where $\widehat{X}_\Phi:=\{ x\in X: \lim\limits_{n\rightarrow
\infty} \frac{1}{n} \varphi_n (x) \mbox{ does not exist} \}$  is
the \emph{irregular set} for the AAP $\Phi=\{\varphi_n\}_{n\ge
1}$.

Given an AAP $\mathcal{F}=\{f_n\}_{n\ge 1}$ and a $T$-invariant
measure $\mu\in \mathcal{M}_T$, let
\[ \mathcal{F}_*(\mu)=\lim\limits_{n\rightarrow\infty}\frac 1 n \int
f_n \mathrm{d}\mu.\]  Using an asymptotically additive potential
$\Phi=\{\varphi_n\}_{n\ge 1}$ as a restrictive condition, this
paper investigates the ergodic optimization of an AAP. Precisely,
  consider two asymptotically additive
potentials $\mathcal{F}=\{f_n\}_{n\ge 1}$ and
$\Phi=\{\varphi_n\}_{n\ge 1}$,
 for any real number $\alpha\in \mathbb{R}$ define
\begin{eqnarray}\label{condi-int}
\Lambda_{\mathcal{F}\mid\Phi}(\alpha):=\sup
\bigr\{\mathcal{F}_*(\mu)\, :\,\mu\in
\mathcal{M}_T(\Phi,\alpha)\bigr\}
\end{eqnarray}
where $\mathcal{M}_T(\Phi,\alpha):=\bigr\{\mu\in
\mathcal{M}_T:\,\Phi_*(\mu)=\alpha \bigr\} $. The quantity
$\Lambda_{\mathcal{F}\mid\Phi}(\alpha)$ is called the {\it
conditional maximum ergodic average}  of the AAP
$\mathcal{F}=\{f_n\}_{n\geq 1}$ (with respect to $\Phi$), and the
following quantity
\[ \beta(\mathcal{F}):=\sup\bigr\{ \mathcal{F}_*(\mu)\ :\ \mu\in
\mathcal{M}_T\bigr\}\] is called the {\it maximum ergodic average}
of the AAP $\mathcal{F}=\{f_n\}_{n\geq 1}$. As a direct
consequence of the main result in \cite{zhao}, we can prove that
\begin{eqnarray}\label{basic-formula}
\beta(\mathcal{F})=\sup_{x\in X} \limsup_{n\rightarrow\infty}\frac
1 n f_n (x) =\lim_{n\rightarrow\infty}\frac 1 n \max_{x\in X}f_n
(x)=\sup_{x\in\mathrm{Reg}(\mathcal{F},T)}\lim_{n\rightarrow\infty}\frac
1 n f_n (x),
\end{eqnarray} where $\mathrm{Reg}(\mathcal{F},T)$ is
the set of points $x\in X$ for which the limit of the sequence
$\{\frac 1 n f_n (x)\}_{n\ge 1}$ exists. Some basic properties of
the function $\Lambda_{\mathcal{F}\mid\Phi}(\cdot)$ are studied,
e.g., the continuity and monotonicity of the function
$\Lambda_{\mathcal{F}\mid\Phi}(\cdot)$. Furthermore, if the TDS
$(X,T)$ satisfies the \emph{specification property} (see the exact
definition  in Section 2), we prove that the maximal growth rate
of an AAP $\mathcal{F}=\{f_n\}_{n\geq 1}$ on  the level sets is
equal to the conditional maximum ergodic average of $\mathcal{F}$,
that is,
\[
\sup_{x\in
K_\Phi(\alpha)}\limsup_{n\rightarrow\infty}\frac{1}{n}f_n(x)=\sup
\bigr\{\mathcal{F}_*(\mu)\ :\ \mu\in
\mathcal{M}_T(\Phi,\alpha)\bigr\},
\]
and the maximum ergodic average of an AAP
$\mathcal{F}=\{f_n\}_{n\geq 1}$ is exactly its maximum
ergodic averages on the irregular sets, i.e.,
\[
\beta(\mathcal{F})=\sup \Big\{\mathcal{F}_*(\mu)\ :\ \mu\in
\bigcup_{x\in \widehat{X}_\Phi}\mathcal{V}(x)
 \Big\}
\]
where $\mathcal{V}(x)$ is the set of all the limit points of the
empirical measures  $\delta_{x,n}:=\frac 1
n\sum_{i=0}^{n-1}\delta_{T^ix}$ and $\delta_x$ is the dirac
measure at the point $x$.
%This result means that the maximal
%integral of an AAP is just determined by its maximal integral on
%the irregular set when the system satisfies specification
%property.
%Furthermore, given a continuous function,
%this paper investigates the relative maximum ergodic average of an
%AAP, see section 2 for the definitions and discussions.

The remainder of this paper is organized as follows. The precise
statements of the main results is given in Section 2. Section 3 provides some examples to illustrate our main results. In section
4, using the methods in the theory of multifractal analysis, we
provides the proofs of all the statement is section 2. In section
5, we apply our main results to suspension flows.

\section{Statements of the results}
This section first provides some preliminaries and notations, and
then gives the statements of the main results of this paper, the
proofs are postponed to Section 4.

 A sequence of continuous functions $\mathcal{F}=\{
f_n\}_{n\geq 1}\subseteq C(X)$ is called an \emph{asymptotically
additive potential}  on $X$, if for each $\xi
>0$, there exists a continuous function $f_\xi\in C(X)$ such that
\begin{eqnarray}\label{Def-aap}
\limsup_{n\rightarrow\infty}\frac{1}{n}||f_n-S_nf_\xi||< \xi.
\end{eqnarray}
This kind of potential was introduced by Feng and Huang \cite{fh}.

Given an AAP $\mathcal{F}=\{f_n\}_{n\geq 1}$, Feng and Huang
\cite{fh} proved the following result.

\begin{theorem}\label{fhthm}Let $\mu$ be a $T-$invariant measure.  Then
the following properties hold:\begin{itemize} \item[(1)] The limit
$\mathcal{F}_*(\mu)=\lim\limits_{n\rightarrow\infty}\frac 1 n \int
f_n \,\mathrm{d}\mu$ exists and is finite. Furthermore, the limit
$\lambda_{\mathcal{F}}(x):=\lim\limits_{n\rightarrow\infty}\frac 1
n f_n(x)$ exists for $\mu-$almost every $x\in X$, and $\int
\lambda_{\mathcal{F}}(x) \,\mathrm{d}\mu=\mathcal{F}_*(\mu)$. In
particular, when $\mu\in \mathcal{E}_T,\
\lambda_{\mathcal{F}}(x)=\mathcal{F}_*(\mu)$ for $\mu-$a.e. $x\in
X$; \item[(2)] The map $\mathcal{F}_*:\mathcal{M}_T\rightarrow
\mathbb{R}$ is continuous; \item[(3)] Let
$\mu=\int_{\mathcal{E}_T} m \,\mathrm{d}\tau(m)$ be the ergodic
decomposition of $\mu\in \mathcal{M}_T$ (here $\tau$ is a probability measure on the space $\mathcal{M}_T$
such that $\tau(\mathcal{E}_T)=1$), then
$\mathcal{F}_*(\mu)=\int_{\mathcal{E}_T} \mathcal{F}_*(m)\,
\mathrm{d}\tau(m)$.\label{dl11}
\end{itemize}
\end{theorem}

Now we recall the definition of the specification property.
Roughly speaking, a TDS $(X,T)$  has specification property if one
can always find a real orbit to interpolate between different
pieces of orbits, up to a pre-assigned error.

\begin{definition} A TDS $(X,T)$ satisfies the specification property if for each
$\epsilon>0$, there exists an integer $m=m(\epsilon)$ such that
for any collection $\{I_j:=[a_j,b_j]\subset \mathbb
N:j=1,2,...,k\}$ of finite intervals with $a_{j+1}-b_j\geq
m(\epsilon)$ for $j=1,2,...,k-1$ and any $x_1,x_2,...,x_k$ in $X$,
there exists a point $x\in X$ such that
\begin{eqnarray} \label{specification}
d(T^{p+a_j}(x),T^p(x_j))<\epsilon
\end{eqnarray}
for all $p=0,1,...,b_j-a_j $ and every $j=1,2,...,k.$
\end{definition}

This definition is weaker than the original Bowen's definition of
specification. In Bowen's definition, it is required that the
shadowing point $x$ is a periodic point of period
$p\geq b_k-a_1+m(\epsilon)$. There are well-know examples of
dynamical systems has the specification property. The most basic
example is the shift map on the full symbolic space; another
example is the topologically mixing shift map on the subshift
space of finite type \cite{dgs}.  Blokh \cite{bl} proved  that a
topologically mixing map of the interval has Bowen's
specification, a recent proof can be found in \cite{bu}.

%Now we introduce some notations. For each $\alpha\in \mathbb{R}$,
%let $\Lambda (\alpha)$ be given as (\ref{condi-int}), and
%$S(\alpha)$ denote the \emph{maximal growth rate} of an AAP
%$\mathcal{F}=\{f_n\}_{n\ge 1}$ on  level sets, i.e.,
%\begin{eqnarray*}
%S(\alpha):=\sup_{x\in
%K_\alpha}\limsup_{n\rightarrow\infty}\frac{1}{n}f_n(x).
%\end{eqnarray*}
%The following quantity
%\[
%I(\mathcal{F}):=\sup \{\mathcal{F}_*(\mu)\ :\ \mu\in \bigcup_{x\in
%X_\Phi}\mathcal{V}(x)
% \}
%\]
% is called the \emph{maximal
%integral} of the AAP on the irregular set.

Given an AAP $\Phi=\{\varphi_n\}_{n\ge 1}$ as the restrictive
condition, let
\begin{eqnarray*}
\eta(\Phi):=\min \bigr\{\Phi_*(\mu):\ \mu\in \mathcal{M}_T\bigr\}
\end{eqnarray*}
be the \emph{minimum ergodic average} of $\Phi$. By (2) of Theorem
\ref{fhthm}, we know that the set
$\mathcal{M}_T(\Phi,\alpha)=\bigr\{\mu\in
\mathcal{M}_T:\,\Phi_*(\mu)=\alpha \bigr\}$ is non-empty,  convex
and compact for each $\alpha\in [\eta(\Phi),\beta(\Phi)]$, so the
 supremum  can be replaced by  maximum  in the definition of
$\Lambda_{\mathcal{F}\mid\Phi}(\alpha)$ in (\ref{condi-int}).

The first theorem says that the function
$\alpha\mapsto\Lambda_{\mathcal{F}\mid\Phi}(\alpha)$ is continuous
on the interval $[\eta(\Phi),\beta(\Phi)]$.

\begin{maintheorem}\label{continuous}Let  $\mathcal{F}=\{f_n\}_{n\geq 1}$ and $\Phi=\{\varphi_n\}_{n\ge 1}$ be two
asymptotically additive potentials on $X$.  Then the map
$\alpha\mapsto\Lambda_{\mathcal{F}\mid\Phi}(\alpha)$ is continuous
on the interval $[\eta(\Phi),\beta(\Phi)]$.
\end{maintheorem}

Since the map $\mathcal{F}_*:\mathcal{M}_T\rightarrow\mathbb{R}$
is continuous and $\mathcal{M}_T$ is weak$^*$ compact, by ergodic
decomposition theorem there exists a $T$-invariant ergodic measure
$\mu^*\in \mathcal{E}_T$ such that
$\beta(\mathcal{F})=\mathcal{F}_*(\mu^*)$. Assume that $
\Phi_*(\mu^*)=\alpha^*$. Note that $\mu^*$ may be not unique
 and
$\beta(\mathcal{F})=\Lambda_{\mathcal{F}\mid\Phi}(\alpha^*)$. The
following theorem studies the monotonicity of the function
$\alpha\mapsto\Lambda_{\mathcal{F}\mid\Phi}(\alpha)$.

\begin{maintheorem}\label{monotone} Let  $\mathcal{F}=\{f_n\}_{n\geq 1}$ and $\Phi=\{\varphi_n\}_{n\ge 1}$ be two
asymptotically additive potentials on $X$. Then the conditional
maximum ergodic average $\Lambda_{\mathcal{F}\mid\Phi}(\cdot)$ is
monotone increasing on $[\eta(\Phi), \alpha^*]$ and monotone
decreasing on $[\alpha^*,\beta(\Phi)]$.
\end{maintheorem}

\begin{remark}Let $\mathcal{M}_{\max}(\mathcal{F}):=\bigr\{\mu\in \mathcal{M}_T\mid
\mathcal{F}_*(\mu)=\beta(\mathcal{F})\bigr\}$ be the set of
maximizing measures for the AAP $\mathcal{F}=\{f_n\}_{n\ge 1}$.
Since the map $\mathcal{F}_*:\mathcal{M}_T\rightarrow \mathbb{R}$
is continuous, the set $\mathcal{M}_{\max}(\mathcal{F})$ is a
non-empty,  convex and compact subset of $\mathcal{M}_T$. Let
\[
\alpha_1:=\min_{\mu\in
\mathcal{M}_{\max}(\mathcal{F})}\Phi_*(\mu)~~~\text{and}~~~\alpha_2:=\max_{\mu\in
\mathcal{M}_{\max}(\mathcal{F})}\Phi_*(\mu).
\]
By Theorem \ref{monotone} we know that
$\Lambda_{\mathcal{F}\mid\Phi}(\alpha)\equiv \beta(\mathcal{F})$
for each $\alpha\in [\alpha_1,\alpha_2]$.
\end{remark}

In the following, we consider the conditional maximum ergodic
average of an AAP $\mathcal{F}=\{f_n\}_{n\geq 1}$ at the extreme
point $\alpha=\beta(\Phi)\ \hbox{or}\ \eta(\Phi)$.

\begin{proposition} \label{extreme}
Let  $\mathcal{F}=\{f_n\}_{n\geq 1}$ and $\Phi=\{\varphi_n\}_{n\ge
1}$ be two asymptotically additive potentials on $X$. Then
$$\Lambda_{\mathcal{F}\mid\Phi}(\alpha)=\sup
\bigr\{\mathcal{F}_*(\mu):\mu\in \mathcal{E}_T\ \hbox{and}\
\Phi_*(\mu)=\alpha\bigr\}$$ when $\alpha=\beta(\Phi)\ \hbox{or}\
\eta(\Phi).$
\end{proposition}

%Let   $\Phi=\{
%\varphi_n\}_{n\geq 1}$  be an  AAP on $X$. For each $\alpha\in \mathbb{R}$ such that $\mathcal{M}_T(\Phi,\alpha)\neq \emptyset$, the following proposition %consider the conditional ergodic average of an AAP $\mathcal{F}=\{ f_n\}_{n\geq 1}$ when the set $\mathcal{M}_T(\Phi,\alpha)$ is perturbed slightly.
%define
%\[
%\mathcal{M}_{\max}(\mathcal{F}\mid
%\Phi,\alpha)=\bigr\{\mu\in\mathcal{M}_T(\Phi,\alpha):\,\mathcal{F}_*(\mu)=\Lambda_{\mathcal{F}\mid\Phi}(\alpha)\bigr\}.
%\]
%We have the following proposition.

%\begin{proposition}\label{relative}Let $(X,T)$ be a TDS,  $ \mathcal{F}=\{f_n\}_{n\geq
%1}$, $\mathcal{G}=\{g_n\}_{n\ge 1}$ and $\Phi=\{\varphi_n\}_{n\ge
%1}$ are  AAPs on $X$.  Assume that  $\alpha$ is a real number such that $\mathcal{M}_T(\Phi,\alpha)\neq\emptyset$ and
%$\mathcal{M}_T(\Phi+\epsilon\mathcal{G},\alpha)\neq\emptyset$ for all sufficiently small $\epsilon>0$. Then
%\[
%\bigr\{ \mathcal{F}_*(\mu) : \ \mu\in
%\mathcal{M}_{T}(\Phi+\epsilon\mathcal{G},\alpha)\bigr \}\rightarrow \bigr\{
%\Lambda_{\mathcal{F}\mid\Phi}(\alpha)\bigr \}\ \ \ \hbox{as}\
%\epsilon\searrow 0
%\]the convergence being in the Hausdorff metric.
%\end{proposition}

For any two asymptotically additive potentials
$\mathcal{F}=\{f_n\}_{n\geq 1}$ and $\Phi=\{\varphi_n\}_{n\ge 1}$,
the following theorem considers the time averages of $\mathcal{F}$
on the level sets of $\Phi$.

\begin{maintheorem}\label{thm.c}Let $(X,T)$ be a TDS satisfying the specification
property, and $\mathcal{F}=\{f_n\}_{n\geq 1}$ and
$\Phi=\{\varphi_n\}_{n\ge 1}$ two asymptotically additive
potentials on $X$.  Then
$$\Lambda_{\mathcal{F}\mid\Phi}(\alpha)=\sup_{x\in
K_\Phi(\alpha)}\limsup_{n\rightarrow\infty}\frac{1}{n}f_n(x)=\sup
\Big\{ \mathcal{F}_*(\mu) :\ \mu\in \bigcup_{x\in
K_\Phi(\alpha)}\mathcal{V}(x)\Big\}$$ for each $\alpha\in
[\eta(\Phi),\beta(\Phi)]$.
\end{maintheorem}

On the irregular set of $\Phi$, we have the following theorem.

\begin{maintheorem}\label{thm.d}Let $(X,T)$ be a TDS satisfying the specification
property,  and $\mathcal{F}=\{f_n\}_{n\geq 1}$ an AAP on $X$.
Assume that $\Phi=\{\varphi_n\}_{n\ge 1}$ is an AAP on $X$
satisfying $\eta(\Phi)<\beta(\Phi)$, then we have
$$\beta(\mathcal{F})=\sup_{x\in\widehat{X}_\Phi}\limsup_{n\to\infty}\frac 1n f_n(x)=
\sup\bigr\{\mathcal{F}_*(\mu):\mu\in \bigcup_{x\in\widehat{X}_\Phi}\mathcal{V}(x)\bigr\}.$$
\end{maintheorem}

\begin{remark}(1)It is easy to see that if $\eta(\Phi)=\beta(\Phi):=\lambda$, then
the sequence $\frac 1 n \varphi_n(x)$  converges uniformly to the
constant $\lambda$. Therefore, in this case there is no irregular
points for the AAP $\Phi=\{\varphi_n\}_{n\ge1}$; (2) Both Theorem
\ref{thm.c} and Theorem \ref{thm.d} use the assumption that the
TDS $(X,T)$ has specification property. Under this assumption we
have $[\eta(\Phi),\beta(\Phi)]=\{\alpha\in \mathbb{R}:
K_\Phi(\alpha)\neq \emptyset\}$ ( see \cite{th1} for a proof),
this result is essentially contained in \cite{tv1}.
\end{remark}

To apply the main results to suspension flows, we consider the
following more general cases of ratios of sequence of AAPs.
Precisely, let $\Phi=\{\varphi_n\}_{n\ge 1}$ and
$\Psi=\{\psi_n\}_{n\ge 1}$ be two asymptotically additive
potentials on $X$. We always assume that
\begin{eqnarray}\label{denom}\frac 1 n\psi_n(x)\ge \sigma~~\forall n\in
\mathbb{N},~\forall x\in X\end{eqnarray} for some constant
$\sigma>0$. For each $\alpha\in \mathbb{R}$, let
\begin{eqnarray*}
E_{\Phi,\Psi}(\alpha):= \{ x\in X  : \lim_{n\rightarrow\infty}
\frac{\varphi_n(x)}{\psi_n(x)}= \alpha \} .
\end{eqnarray*}
The ergodic optimization  on this level set has a similar property
as Theorem \ref{thm.c}.

\begin{maintheorem}\label{thm.e}Let $(X,T)$ be a TDS satisfying the specification
property,  $ \mathcal{F}=\{f_n\}_{n\geq 1}$  and
$\Phi=\{\varphi_n\}_{n\ge 1}$ are AAPs on $X$. Assume that
$\mathcal{G}=\{g_n\}_{n\ge 1}$ and $\Psi=\{\psi_n\}_{n\ge 1}$ are
two AAPs satisfying (\ref{denom}). Then, for  any real number
$\alpha$ with $E_{\Phi,\Psi}(\alpha)\neq \emptyset $
\begin{eqnarray*}
\sup_{x\in E_{\Phi,\Psi}(\alpha)}\limsup_{n\rightarrow\infty}
\frac{f_n(x)}{g_n(x)}&=&\sup \Big\{
\frac{\mathcal{F}_*(\mu)}{\mathcal{G}_*(\mu)}:\ \mu\in
\mathcal{M}_T\ \hbox{and}\ \frac{\Phi_*(\mu)}{\Psi_*(\mu)}=\alpha\Big\}\\
&=&\sup \Big\{ \frac{\mathcal{F}_*(\mu)}{\mathcal{G}_*(\mu)}:\
\mu\in \bigcup_{x\in E_{\Phi,\Psi}(\alpha) }\mathcal{V}(x)\Big\}.
\end{eqnarray*}
\end{maintheorem}

\begin{remark}\label{weak-spec} In the proof of Theorem \ref{thm.e},  we only use the fact that the generic set $G_{\mu}$ is
non-empty for each $\mu\in \mathcal{M}_T$, where
$G_{\mu}=\bigr\{x\in X:
\delta_{x,n}\rightarrow\mu\,(n\rightarrow\infty)\bigr\}$. This
property is always called \emph{saturated property} of a TDS
$(X,T)$, and is satisfied when a TDS $(X,T)$ satisfying
specification property. This observation means that Theorem
\ref{thm.e} remains true for a broader class of systems. For
example, Pfister and Sullivan \cite{ps} consider a weak
specification property which is called the \emph{$g-$almost
product property} in that paper.  In \cite{ps}, they proved that a TDS $(X,T)$
satisfies the saturated property when it has $g-$almost product
property. Therefore, Theorem \ref{thm.e} remains true if a TDS
$(X,T)$ has $g-$almost product property.
\end{remark}

Let $\Phi=\{\varphi_n\}_{n\ge 1}$ and $\Psi=\{\psi_n\}_{n\ge 1}$
be two asymptotically additive potentials given as above,  the
irregular set associated with $\Phi$ and $\Psi$ is defined as
\[
\widehat{X}_{\Phi,\Psi}:=\bigr\{x\in X: \lim_{n\rightarrow\infty}
\frac{\varphi_n(x)}{\psi_n(x)}\ \hbox{does not exist}\bigr\}.
\]

\begin{maintheorem}\label{thm.f}Let $(X,T)$ be a TDS satisfying the specification
property,  $ \mathcal{F}=\{f_n\}_{n\geq 1}$  and
$\Phi=\{\varphi_n\}_{n\ge 1}$ are AAPs on $X$. Assume that
$\mathcal{G}=\{g_n\}_{n\ge 1}$ and $\Psi=\{\psi_n\}_{n\ge 1}$ are
two AAPs satisfying (\ref{denom}), and
$\inf\limits_{\mu\in\mathcal{M}_T}\frac{\Phi_*(\mu)}{\Psi_*(\mu)}<\sup\limits_{\mu\in\mathcal{M}_T}\frac{\Phi_*(\mu)}{\Psi_*(\mu)}$,
then we have
\begin{eqnarray*}
\sup \Big\{\frac{\mathcal{F}_{*}(\mu)}{\mathcal{G}_{*}(\mu)}:\
\mu\in \mathcal{E}_T\Big\}&=&\sup
\Big\{\frac{\mathcal{F}_{*}(\mu)}{\mathcal{G}_{*}(\mu)}:\ \mu\in
\mathcal{M}_T\Big\}\\
&=&\sup
\Big\{\frac{\mathcal{F}_{*}(\mu)}{\mathcal{G}_{*}(\mu)}:\ \mu\in
\bigcup_{x\in\widehat{X}_{\Phi,\Psi}}\mathcal{V}(x)\Big\}\\
&=&\sup_{x\in \widehat{X}_{\Phi,\Psi}}\limsup_{n\to\infty}\frac{f_n(x)}{g_n(x)}.
\end{eqnarray*}
\end{maintheorem}

 If the TDS $(X,T)$ satisfying the $g-$almost product property (see \cite{ps} for the definition),
then the  proof of Theorem \ref{thm.f} generalizes
unproblematically to this setting and thus Theorem \ref{thm.f}
holds for continuous maps with the g-almost product property.

\section{Examples}
This section provides examples to illustrate the main results in Section 2.

We first provide a system that does satisfy weak
specification property but does not satisfy the  specification property, however, the main results are still valid for this kind of systems.
\begin{example}
Consider the piecewise expanding maps of the interval $[0,1)$
given by $T_\beta(x)=\beta x (\!\!\!\mod 1)$, where $\beta>1$.
This family is known as \emph{beta transformations} and it was
introduced by R\'enyi in \cite{re}. From \cite{bu}, we know that
for all but countable many values of $\beta$ the transformation
$T_\beta$ does not satisfy the specification property.  However,
it follows from \cite{ps,Th10} that every $\beta$-transformation  satisfies
the \emph{$g$-almost product property}. By Remark \ref{weak-spec},
Theorem \ref{thm.c}--\ref{thm.f} are true for
$\beta$-transformation for every $\beta>1$.
\end{example}

The following particular asymptotically additive potentials were introduced by Barreira \cite{ba06} and Mummert \cite{mu06} independently to study the theory of thermodynamic formalism for a broader  class of potentials.
\begin{example}Let $\mathcal{F}=\{f_n\}_{n\ge 1}$  be an almost additive potential, i.e, there exists $C>0$ such that
$ f_n+f_m\circ T^n-C\le f_{n+m}\le f_n+f_m\circ T^n+C$ for any $n,m\in \mathbb{N}$. This kind of potential was introduced by Barreira \cite{ba06} and Mummert \cite{mu06} independently. They independently investigated the theory of thermodynamic formalism for almost additive potentials, including the existence and uniqueness of equilibrium state and Gibbs state, variational principle for almost additive topological pressure.

It is not hard to see that an almost additive potential is asymptotically additive, see \cite[Proposition A.5]{fh} or \cite[Proposition 2.1]{zzc} for proofs. Hence, Theorem \ref{continuous}--\ref{thm.f} are true for any asymptotically additive potentials.
\end{example}

The last example deals with  families of potentials
responsible for computing the largest and smallest Lyapunov exponents are asymptotically additive.

\begin{example}
Let $M$ be a $d$-dimensional smooth  manifold and $J$ a compact
expanding invariant set for a $C^1$ map $f$. Let
$\mathcal{E}(f\mid_J)$ denote the set of all $f$-invariant ergodic
measures supported on $J$.  We say that $J$ is an average
conformal repeller if all Lyapunov exponents of each ergodic
measure $\mu\in \mathcal{E}(f\mid_J)$ are equal and positive. In
particular, it follows from \cite[Theorem~4.2]{BCH10} that
\begin{equation}\label{acr-uniform}
\begin{aligned}
\lim_{n\to\infty} \frac1n  \Big( \log \|Df^n(x)\| - \log
\|Df^n(x)^{-1}\|^{-1} \Big)
    = \lim_{n\to\infty} \frac1n  \log \frac{\|Df^n(x)\|}{\|Df^n(x)^{-1}\|^{-1}}
    = 0
\end{aligned}
\end{equation}
uniformly on $J$. Let $\Psi_1=\{\log \|Df^n(x)\|\}_{n\ge 1}$ and  $\Psi_2=\{\log \|Df^n(x)^{-1}\|^{-1}\}_{n\ge 1}$,  it is not hard to
check that these two  potentials are asymptotically additive since
they can be  approximated by the additive potentials
$\{\frac 1d\log |\det(Df^n(x))|\}_{n\ge 1}$. Note that the potential
$\Psi_3=\bigr\{\log \frac{\|Df^n(x)\|}{\|Df^n(x)^{-1}\|^{-1}}\bigr\}_{n\ge 1}$
is also asymptotically additive.

Assume further that $J$ is topological mixing, then $(J,f)$
satisfies the specification property. The following facts hold for
the expanding system $(J,f)$:
\begin{itemize}
\item[(i)] By \eqref{acr-uniform} there is no irregular point for
the AAP $\Psi_3$; \item[(ii)] Given any AAP
$\Phi=\{\varphi_n\}_{n\ge 1}$ satisfying $\eta(\Phi)<\beta(\Phi)$,
the following properties hold:
\begin{itemize}
\item[(1)]by Theorem \ref{thm.c} and \eqref{acr-uniform}, for each
$\alpha\in [\eta(\Phi),\beta(\Phi)]$ we have
\begin{eqnarray*}&&\Lambda_{\Psi_1\mid\Phi}(\alpha)=\sup\limits_{x\in
K_\Phi(\alpha)}\limsup\limits_{n\rightarrow\infty}\frac{1}{n}\log
\|Df^n(x)\|\\&&=\sup \Big\{ \Psi_{1*}(\mu) :\ \mu\in
\bigcup\limits_{x\in
K_\Phi(\alpha)}\mathcal{V}(x)\Big\}=\Lambda_{\Psi_2\mid\Phi}(\alpha);\end{eqnarray*}
\item[(2)] by \eqref{basic-formula}, Theorem \ref{thm.d} and \eqref{acr-uniform} we have
\begin{eqnarray*}
\sup\bigr\{\Psi_{1*}(\mu):\mu\in
\bigcup_{x\in\widehat{X}_\Phi}\mathcal{V}(x)\bigr\}&=&\sup\bigr\{\Psi_{2*}(\mu):\mu\in
\bigcup_{x\in\widehat{X}_\Phi}\mathcal{V}(x)\bigr\}\\
&=&\sup_{x\in\widehat{X}_\Phi }\limsup_{n\rightarrow\infty} \frac1n \log
\|Df^n(x)\|\\
&=&\sup_{x\in J}\limsup_{n\rightarrow\infty} \frac1n \log
\|Df^n(x)\|;
\end{eqnarray*}
\end{itemize}
\item[(iii)]Let $\Psi_3$ be the restrictive condition, by \eqref{acr-uniform}, Theorems \ref{thm.c} and
\ref{thm.d}, for any asymptotically additive potential
$\mathcal{G}=\{g_n\}_{n\ge 1}$ we have
\begin{eqnarray*}\Lambda_{\mathcal{G}\mid\Psi_3}(0)=\sup\limits_{x\in
K_{\Psi_3}(0)}\limsup\limits_{n\rightarrow\infty}\frac{1}{n}g_n(x)=\sup
\Big\{ \mathcal{G}_{*}(\mu) :\ \mu\in \bigcup\limits_{x\in
K_{\Psi_3}(0)}\mathcal{V}(x)\Big\}=\beta(\mathcal{G}).\end{eqnarray*}
\end{itemize}
\end{example}

\section{Proofs}
This section provides the proofs of the theorems stated in Section
2.

 \begin{proof}[Proof of Theorem \ref{continuous}] Assume that
$\eta(\Phi)=\Phi_*(\mu_1)$ and $\beta(\Phi)=\Phi_*(\mu_2)$ for
some invariant measures $\mu_1,\mu_2\in\mathcal{M}_T $. Let
$\{\alpha_n\}_{n\ge 1}\subseteq [\eta(\Phi),\beta(\Phi)]$ so that
$\alpha_n\rightarrow \alpha$ as $n\rightarrow\infty$, to prove the
continuity of the map $\Lambda_{\mathcal{F}\mid \Phi}(\cdot)$ it
suffices to prove that
$$\lim\limits_{n\rightarrow\infty}\Lambda_{\mathcal{F}\mid \Phi}(\alpha_n)=\Lambda_{\mathcal{F}\mid \Phi}(\alpha).$$

For each $n\ge 1$, since $\mathcal{M}_T(\Phi,\alpha_n)$ is
non-empty and compact, there exists an invariant measure $\mu_n\in
\mathcal{M}_T(\Phi,\alpha_n)$ such that $\Lambda_{\mathcal{F}\mid
\Phi}(\alpha_n)=\mathcal{F}_{*}(\mu_n)$. Choose  a subsequence of
integers $\{n_j\}_{j\geq 1}$  such that
$$\limsup\limits_{n\rightarrow\infty}\Lambda_{\mathcal{F}\mid \Phi}(\alpha_n)=\lim\limits_{j\rightarrow\infty}\Lambda_{\mathcal{F}\mid \Phi}(\alpha_{n_j})$$
and  $\mu_{n_j}\rightarrow \mu\,(j\rightarrow\infty)$ for some
$\mu\in \mathcal{M}_T$. The strategy for the proof is to
approximate the  asymptotically additive potential by
appropriate Birkhoff sums associated with some continuous
function. Indeed, fix a small number $\xi>0$ and $\varphi_\xi$ is
a continuous function given by (\ref{Def-aap}) approximating
$\Phi$. We have
\[
\int\varphi_\xi\, \mathrm{d}\mu =\lim_{j\rightarrow\infty}
\int\varphi_\xi\, \mathrm{d}\mu_{n_j}.
\]
Therefore, for all sufficiently large $j$ the following holds:
\begin{eqnarray*}
\Big| \Phi_*(\mu)-\alpha\Big|&\le&
\Big|\Phi_*(\mu)-\int\varphi_\xi\,
\mathrm{d}\mu\Big|+\Big|\int\varphi_\xi\,
\mathrm{d}\mu-\int\varphi_\xi\,
\mathrm{d}\mu_{n_j}\Big|+\Big|\int\varphi_\xi\,
\mathrm{d}\mu_{n_j}-\alpha\Big|\\
&\le& \xi+\xi+\Big|\int\varphi_\xi\,
\mathrm{d}\mu_{n_j}-\alpha\Big|.
\end{eqnarray*}
Since  $\mu_{n_j}\in \mathcal{M}_T(\Phi,\alpha_{n_j})$, we have
\begin{eqnarray*}
\Big|\int\varphi_\xi\, \mathrm{d}\mu_{n_j}-\alpha\Big|&\le&
\Big|\int\varphi_\xi\,
\mathrm{d}\mu_{n_j}-\Phi_*(\mu_{n_j})\Big|+\Big|\Phi_*(\mu_{n_j})-\alpha\Big|\\
&\le&\xi+\bigr|\alpha_{n_j}-\alpha\bigr|\\
&\le&2\xi
\end{eqnarray*}
for all sufficiently large $j$. The above two inequalities yield
that
\[
\Big| \Phi_*(\mu)-\alpha\Big|\le 4\xi.
\]
 The arbitrariness of $\xi$ implies that  $\mu\in
\mathcal{M}_T(\Phi,\alpha)$. Therefore
\begin{eqnarray}\label{usc}
\limsup_{n\rightarrow\infty}\Lambda_{\mathcal{F}\mid\Phi}(\alpha_n)=\lim_{j\rightarrow\infty}\Lambda_{\mathcal{F}\mid\Phi}(\alpha_{n_j})=\lim_{j\rightarrow\infty}\mathcal{F}_{*}(\mu_{n_j})=\mathcal{F}_{*}(\mu)\leq
\Lambda_{\mathcal{F}\mid\Phi}(\alpha).
\end{eqnarray}

On the other hand, choose $\mu\in \mathcal{M}_T(\Phi,\alpha)$ such
that $\mathcal{F}_{*}(\mu)=\Lambda_{\mathcal{F}\mid\Phi}(\alpha)$.
For each $n\ge 1$, if $\eta(\Phi)\le \alpha_n\le \alpha$, take a
$T-$invariant measure $\mu_p$ of the form $p\mu_1+(1-p)\mu$ such
that $\Phi_*(\mu_p)=\alpha_n$. If $\alpha\le
\alpha_n\le\beta(\Phi)$, consider a $T-$invariant measure $\mu_p$
of the form $p\mu_2+(1-p)\mu$ such that $\Phi_*(\mu_p)=\alpha_n$.
Note that $p\rightarrow 0$ as $n\rightarrow\infty$, and
\[
\mathcal{F}_{*}(\mu_p)=p\mathcal{F}_{*}(\mu_1)+(1-p)\mathcal{F}_{*}(\mu)~\text{or}~p\mathcal{F}_{*}(\mu_2)+(1-p)\mathcal{F}_{*}(\mu).
\]
Hence,
\begin{eqnarray}\label{lsc}
\liminf_{n\rightarrow\infty}\Lambda_{\mathcal{F}\mid\Phi}(\alpha_n)\geq
\lim_{p\rightarrow
0}\mathcal{F}_{*}(\mu_p)=\mathcal{F}_{*}(\mu)=\Lambda_{\mathcal{F}\mid\Phi}(\alpha).
\end{eqnarray}
Combining inequalities (\ref{usc}) and (\ref{lsc}), the desired
result immediately follows.
\end{proof}

\begin{remark}\label{asp}Note that in the proof of the inequalities (\ref{usc}) and (\ref{lsc}),  it
is enough to require that the map $\mu\mapsto\mathcal{F}_{*}(\mu)$
is upper-semi-continuous.  Therefore,  Theorem \ref{continuous} is
also true for \emph{asymptotically sub-additive potentials}.
  A sequence $\mathcal{F}=\{
f_n\}_{n\geq 1}\subseteq C(X)$ is called an asymptotically
sub-additive potential on $X$, if for each $\xi
>0$ there exists a sub-additive potential
$\Phi^\xi=\{\varphi_{n}^{\xi}\}_{n\geq 1}\subset C(X)$, i.e.
$\varphi_{n+m}^{\xi}\leq
\varphi_{n}^{\xi}+\varphi_{m}^{\xi}\circ T^n$  for any $x\in X$ and any
$n,m\in \mathbb{N}$, such that
\[
\limsup_{n\rightarrow\infty}\frac{1}{n}||f_n-\varphi_{n}^{\xi}||<
\xi.
\]
 In this case, it was proved by
Feng and Huang(cf. \cite{fh}) that the map
$\mu\mapsto\mathcal{F}_{*}(\mu)$ is
upper-semi-continuous.\end{remark}

\begin{proof}[Proof of Theorem \ref{monotone}]
We first show that $\Lambda_{\mathcal{F}\mid\Phi}(\alpha)$ is
monotone increasing on $[\eta(\Phi), \alpha^*]$. Let
$\alpha_1,\alpha_2\in [\eta(\Phi), \alpha^*]$ and
$\alpha_1<\alpha_2$, choose a $T-$invariant measure $\nu\in
\mathcal{M}_T(\Phi,\alpha_1)$  so that
$\mathcal{F}_{*}(\nu)=\Lambda_{\mathcal{F}\mid\Phi}(\alpha_1)$.
For $t\in [0,1]$, put $\mu_t=(1-t)\nu+t\mu^*$, then we have
$$\mathcal{F}_{*}(\mu_t)=(1-t)\Lambda_{\mathcal{F}\mid\Phi}(\alpha_1)+t\Lambda_{\mathcal{F}\mid\Phi}(\alpha^*).$$
Since
$\Lambda_{\mathcal{F}\mid\Phi}(\alpha^*)=\beta(\mathcal{F})$, we
have
$$\frac{\mathrm{d}}{\mathrm{d}t}\mathcal{F}_{*}(\mu_t)=\Lambda_{\mathcal{F}\mid\Phi}(\alpha^*)-\Lambda_{\mathcal{F}\mid\Phi}(\alpha_1)\geq
0.$$ Therefore, the map $\mathcal{F}_{*}(\mu_t)$ is monotone
increasing w.r.t. $t$ on the interval $[0,1]$. Since
$\alpha_1<\alpha_2\le\alpha^*$,  there exists $t_0\in [0,1]$ such
that $$\Phi_*(\mu_{t_0})=(1-t_0)\alpha_1+t_0\alpha^*=\alpha_2.$$
Hence,
\begin{eqnarray*}
\Lambda_{\mathcal{F}\mid\Phi}(\alpha_1)=\mathcal{F}_{*}(\mu_0)\leq
\mathcal{F}_{*}(\mu_{t_0})\leq
\Lambda_{\mathcal{F}\mid\Phi}(\alpha_2)
\end{eqnarray*}
which shows that $\Lambda_{\mathcal{F}\mid\Phi}(\alpha)$ is
monotone increasing on $[\eta(\Phi), \alpha^*]$.

Next we prove that $\Lambda_{\mathcal{F}\mid\Phi}(\alpha)$ is
monotone decreasing on $[\alpha^*,\beta(\Phi)]$. The methods used
here are similar as the above arguments. Let $\alpha_3,\alpha_4\in
[\alpha^*,\beta(\Phi)]$ and $\alpha_3<\alpha_4$,  choose a
$T-$invariant measure $m\in \mathcal{M}_T(\Phi,\alpha_4)$ so that
$\mathcal{F}_{*}(m)=\Lambda_{\mathcal{F}\mid\Phi}(\alpha_4)$. For
$t\in [0,1]$, put $\nu_t=(1-t)\mu^*+tm$, then we have
$$\mathcal{F}_{*}(\nu_t)=(1-t)\Lambda_{\mathcal{F}\mid\Phi}(\alpha^*)+t\Lambda_{\mathcal{F}\mid\Phi}(\alpha_4).$$
Since
$$\frac{\mathrm{d}}{\mathrm{d}t}\mathcal{F}_{*}(\nu_t)=\Lambda_{\mathcal{F}\mid\Phi}(\alpha_4)-\Lambda_{\mathcal{F}\mid\Phi}(\alpha^*)\leq
0,$$ the map $\mathcal{F}_{*}(\nu_t)$ is monotone decreasing
w.r.t. $t$ on the interval $[0,1]$.  And since
$\alpha^*\le\alpha_3<\alpha_4$, there exists $t_1\in [0,1]$ such
that
$$\Phi_*(\nu_{t_1})=(1-t_1)\alpha^*+t_1\alpha_4=\alpha_3.$$ Hence,
\begin{eqnarray*}
\Lambda_{\mathcal{F}\mid\Phi}(\alpha_4)=\mathcal{F}_{*}(\nu_1)\leq
\mathcal{F}_{*}(\nu_{t_1})\leq
\Lambda_{\mathcal{F}\mid\Phi}(\alpha_3).
\end{eqnarray*}
This shows that $\Lambda_{\mathcal{F}\mid\Phi}(\alpha)$ is
monotone decreasing on $[\alpha^*,\beta(\Phi)]$.
\end{proof}
The same reason as presented in Remark \ref{asp}, Theorem
\ref{monotone} is also true for asymptotically sub-additive
potentials.

\begin{proof}[Proof of Proposition \ref{extreme}]
This proof only deals with the case $\alpha=\beta(\Phi)$, since
the
 other case of $\alpha=\eta(\Phi)$ can be proven in a similar
fashion. Since $\mathcal{M}_T(\Phi,\beta(\Phi))$ is a compact
subset of $\mathcal{M}_T$, there exists a $T$-invariant measure
$\mu\in\mathcal{M}_T(\Phi,\beta(\Phi))$ such that
$\mathcal{F}_*(\mu)=\Lambda_{\mathcal{F}\mid\Phi}(\beta(\Phi))$.
Let $\mu=\int_{\mathcal{E}_T} m \,\mathrm{d}\tau(m)$ be the
ergodic decomposition of $\mu$, then there exists a full
$\tau-$measure set $\Omega\subset \mathcal{E}_T$ such that
$\Phi_*(m)=\beta(\Phi)$ for each $m\in \Omega$. Using (3) of
Theorem \ref{fhthm}, we have
\begin{eqnarray*}
\Lambda_{\mathcal{F}\mid\Phi}(\beta(\Phi))=\mathcal{F}_*(\mu)=\int_{\mathcal{E}_T}
\mathcal{F}_*(m)\,\mathrm{d}\tau(m)=\int_{\Omega}
\mathcal{F}_*(m)\,\mathrm{d}\tau(m).
\end{eqnarray*}
Notice that $\mathcal{F}_*(m)\le
\Lambda_{\mathcal{F}\mid\Phi}(\beta(\Phi))$ for each $T$-invariant
ergodic measure $m\in \Omega$, therefore there must exists some
$T$-invariant ergodic measure $m\in \Omega$ such that
$\Lambda_{\mathcal{F}\mid\Phi}(\beta(\Phi))=\mathcal{F}_*(m)$.
This completes the proof.
\end{proof}

\begin{proof}[Proof of Theorem \ref{thm.c}] We consider the
asymptotically additive potentials $g_n(x)\equiv n$ and
$\psi_n(x)\equiv n$ for each $x\in X$ and $n\in \mathbb{N}$, then
Theorem \ref{thm.c} is a direct consequence of Theorem
\ref{thm.e}.
\end{proof}

\begin{proof}[Proof of Theorem \ref{thm.d}] In Theorem
\ref{thm.f}, we consider the particular asymptotically additive
potentials $g_n(x)\equiv n$ and $\psi_n(x)\equiv n$ for each $x\in
X$ and $n\in \mathbb{N}$, then Theorem \ref{thm.d} is a direct
consequence of Theorem \ref{thm.f}.
\end{proof}

\begin{proof}[Proof of Theorem \ref{thm.e}]  We divide the proof  into
 two parts.

 {\bf Part I:} For any real number $\alpha\in \mathbb{R}$ such
 that $E_{\Phi,\Psi}(\alpha)\neq\emptyset$, we will show that
 \[
\sup \Big\{ \frac{\mathcal{F}_*(\mu)}{\mathcal{G}_*(\mu)}:\ \mu\in
\mathcal{M}_T\ \hbox{and}\
\frac{\Phi_*(\mu)}{\Psi_*(\mu)}=\alpha\Big\} =\sup \Big\{
\frac{\mathcal{F}_*(\mu)}{\mathcal{G}_*(\mu)}:\ \mu\in
\bigcup_{x\in E_{\Phi,\Psi}(\alpha) }\mathcal{V}(x)\Big\}.
 \]

 Choose a $T$-invariant measure $\mu\in \mathcal{M}_T$ such that $\frac{\Phi_*(\mu)}{\Psi_*(\mu)}=\alpha.$ Since the TDS  $(X,T)$
satisfies the specification property, there exists some point
$x_0\in G_\mu$, i.e.,
$\delta_{x_0,n}\rightarrow\mu$ as $n\rightarrow\infty$. This means
that $\mathcal{V}(x_0)=\{\mu\}$. Fix a small number $\xi>0$, let
 $\varphi_\xi$ and $\psi_\xi$ be continuous
functions given by (\ref{Def-aap}) approximating $\Phi$ and $\Psi$
respectively. Note that
\[
\lim_{n\rightarrow\infty}\frac 1 n S_n\varphi_\xi(x_0)=\int
\varphi_\xi\,\mathrm{d}\mu~~\text{and}~~\lim_{n\rightarrow\infty}\frac
1 n S_n\psi_\xi(x_0)=\int \psi_\xi\,\mathrm{d}\mu.
\]
Therefore, for all sufficiently large $n$ we have that
\begin{eqnarray*}
\frac{\varphi_n(x_0)}{\psi_n(x_0)}\le
\frac{S_n\varphi_\xi(x_0)+\xi}{S_n\psi_\xi(x_0)-\xi}\le \frac{\int
\varphi_\xi\,\mathrm{d}\mu+2\xi}{\int
\psi_\xi\,\mathrm{d}\mu-2\xi}\le
\frac{\Phi_*(\mu)+3\xi}{\Psi_*(\mu)-3\xi}.
\end{eqnarray*}
Similarly, for all sufficiently large $n$ we obtain
\[
\frac{\varphi_n(x_0)}{\psi_n(x_0)}\ge\frac{\Phi_*(\mu)-3\xi}{\Psi_*(\mu)+3\xi}.
\]
 The above two inequalities yield that
\[
\lim_{n\rightarrow\infty}\frac{\varphi_n(x_0)}{\psi_n(x_0)}=\frac{\Phi_*(\mu)}{\Psi_*(\mu)}=\alpha.
\]
Hence, we know that $x_0\in E_{\Phi,\Psi}(\alpha)$. In consequence
we have that $\mu\in \bigcup\limits_{x\in
E_{\Phi,\Psi}(\alpha)}\mathcal{V}(x)$. This yields that
\[
\sup \Big\{ \frac{\mathcal{F}_*(\mu)}{\mathcal{G}_*(\mu)}:\ \mu\in
\mathcal{M}_T\ \hbox{and}\
\frac{\Phi_*(\mu)}{\Psi_*(\mu)}=\alpha\Big\}\le\sup \Big\{
\frac{\mathcal{F}_*(\mu)}{\mathcal{G}_*(\mu)}:\ \mu\in
\bigcup_{x\in E_{\Phi,\Psi}(\alpha) }\mathcal{V}(x)\Big\}.
 \]

Conversely, given $\mu\in \mathcal{V}(x)$ for some $x\in
E_{\Phi,\Psi}(\alpha)$, we may assume that there exists a
subsequence of positive integers $\{n_k\}_{k\geq 1}$ such that
$\delta_{x,n_k}\rightarrow\mu$ as $k\rightarrow\infty$. Clearly, $\mu$
is a $T$-invariant measure.  Fix a small number $\xi>0$, let
 $\varphi_\xi$ and $\psi_\xi$ be continuous
functions given by (\ref{Def-aap}) approximating $\Phi$ and $\Psi$
respectively. Then we have
\[
\int \varphi_\xi\, \mathrm{d}\mu=\lim_{k\rightarrow\infty} \int
\varphi_\xi\,
\mathrm{d}\delta_{x,n_k}=\lim_{k\rightarrow\infty}\frac 1 n_k
S_{n_k}\varphi_\xi(x)
\]
and
\[
\int \psi_\xi\, \mathrm{d}\mu=\lim_{k\rightarrow\infty} \int
\psi_\xi\, \mathrm{d}\delta_{x,n_k}=\lim_{k\rightarrow\infty}\frac
1 n_k S_{n_k}\psi_\xi(x).
\]
Hence, for all sufficiently lareg $k$ we have
\begin{eqnarray*}
\frac{\Phi_*(\mu)}{\Psi_*(\mu)}&\le& \frac{\int \varphi_\xi\,
\mathrm{d}\mu+\xi}{\int \psi_\xi\,
\mathrm{d}\mu-\xi}\le\frac{\frac{1}{n_k}
S_{n_k}\varphi_\xi(x)+2\xi}{\frac{1}{n_k} S_{n_k}\psi_\xi(x)-2\xi}\\
&\le&\frac{\varphi_{n_k}(x)+3\xi}{\psi_{n_k}(x)-3\xi}\le
\alpha+5\xi,
\end{eqnarray*}
the last inequality holds since  $x\in E_{\Phi,\Psi}(\alpha)$.
Similarly, we can obtain that
\[
\frac{\Phi_*(\mu)}{\Psi_*(\mu)}\ge\alpha-5\xi.
\]
The arbitrariness of $\xi$ implies that
$$\frac{\Phi_*(\mu)}{\Psi_*(\mu)}=\alpha.$$
This yields that \[ \sup \Big\{
\frac{\mathcal{F}_*(\mu)}{\mathcal{G}_*(\mu)}:\ \mu\in
\mathcal{M}_T\ \hbox{and}\
\frac{\Phi_*(\mu)}{\Psi_*(\mu)}=\alpha\Big\}\ge\sup \Big\{
\frac{\mathcal{F}_*(\mu)}{\mathcal{G}_*(\mu)}:\ \mu\in
\bigcup_{x\in E_{\Phi,\Psi}(\alpha) }\mathcal{V}(x)\Big\}.
 \]
Combining the above arguments, the desired result immediately
follows.

{\bf Part II:} In this part, we will show that
\[
\sup_{x\in
E_{\Phi,\Psi}(\alpha)}\limsup_{n\rightarrow\infty}\frac{f_n(x)}{g_n(x)}=\sup
\Big\{ \frac{\mathcal{F}_*(\mu)}{\mathcal{G}_*(\mu)}:\ \mu\in
\mathcal{M}_T\ \hbox{and}\
\frac{\Phi_*(\mu)}{\Psi_*(\mu)}=\alpha\Big\}.
\]

The strategy of the proof  is to approximate the asymptotically
additive potentials by appropriate Birkhoff sums associated with
some continuous function.

Given a small number $\xi>0$, since $\mathcal{F}=\{
f_n\}_{n=1}^{\infty}$ is an AAP, there exists a continuous
function  $f_\xi\in C(X)$  such that
\begin{eqnarray}\label{ds34}
\frac 1 n S_nf_\xi(x) -\xi\leq \frac 1 n f_n(x)\leq \frac 1 n
S_nf_\xi(x) +\xi,\ \ \forall x\in X
\end{eqnarray}
for all sufficiently large $n$. This yields that
\begin{eqnarray}\label{approx}
\mathcal{F}_*(\mu)=\lim_{n\rightarrow\infty}\frac 1 n \int
f_n(x)\, \mathrm{d}\mu=\lim_{\xi\rightarrow 0} \int f_\xi(x) \,
\mathrm{d}\mu
\end{eqnarray}
for each $T$-invariant measure $\mu\in \mathcal{M}_T$. Similarly,
for the AAP $\mathcal{G}=\{g_n\}_{n\ge 1}$, there exists a
continuous function $g_\xi$ such that the corresponding properties
(\ref{ds34}) and (\ref{approx}) hold.

Fix a number $\alpha\in \mathbb{R}$ such that
$E_{\Phi,\Psi}(\alpha)\neq\emptyset$. Take $x\in
E_{\Phi,\Psi}(\alpha)$ and choose a subsequence of positive
integers $\{n_j\}_{j\geq 1}$ such that the following properties
hold:
\begin{itemize}
 \item[(i)] $\lim\limits_{j\rightarrow\infty}\frac
{f_{n_j}(x)}{g_{n_j}(x)}
=\limsup\limits_{n\rightarrow\infty}\frac{f_{n}(x)}{g_n(x)}$;
\item[(ii)] $\delta_{x,n_j}\rightarrow \mu$ for some $\mu\in
\mathcal{M}_T$ as $j\rightarrow\infty$.\end{itemize} Take the same
$\xi>0$ as above, and $\varphi_\xi$ and $\psi_\xi$ are continuous
functions given by (\ref{Def-aap}) approximating $\Phi$ and $\Psi$
respectively. Note that
\[
\int \varphi_\xi\, \mathrm{d}\mu=\lim_{j\rightarrow\infty} \int
\varphi_\xi\,
\mathrm{d}\delta_{x,n_j}=\lim_{j\rightarrow\infty}\frac{1}{n_j}
S_{n_j}\varphi_\xi(x)
\]
and
\[
\int \psi_\xi\, \mathrm{d}\mu=\lim_{j\rightarrow\infty} \int
\psi_\xi\,
\mathrm{d}\delta_{x,n_j}=\lim_{j\rightarrow\infty}\frac{1}{n_j}
S_{n_j}\psi_\xi(x).
\]
Therefore, for all sufficiently large $j$ we have
\begin{eqnarray*}
\frac{\Phi_*(\mu)}{\Psi_*(\mu)}&\le& \frac{\int \varphi_\xi\,
\mathrm{d}\mu+\xi}{\int \psi_\xi\,
\mathrm{d}\mu-\xi}\le\frac{\frac{1}{n_j}
S_{n_j}\varphi_\xi(x)+2\xi}{\frac{1}{n_j} S_{n_j}\psi_\xi(x)-2\xi}\\
&\le&\frac{\varphi_{n_j}(x)+3\xi}{\psi_{n_j}(x)-3\xi}\le
\alpha+5\xi,
\end{eqnarray*}
where the last inequality holds since $x\in
E_{\Phi,\Psi}(\alpha)$. Similarly, we can prove that
\[
\frac{\Phi_*(\mu)}{\Psi_*(\mu)}\ge \alpha-5\xi.
\]
By the arbitrariness of $\xi$, we have
\[
\frac{\Phi_*(\mu)}{\Psi_*(\mu)}= \alpha.
\]

On the other hand, note that
\[
\frac{1}{n_j} \|f_{n_j}(x)- S_{n_j}f_\xi\|
<\xi~~\text{and}~~\frac{1}{n_j} \|g_{n_j}(x)- S_{n_j}g_\xi\| <\xi
\]
for all sufficiently large $j$. Using properties (i) and (ii),  we
have that
\begin{eqnarray*}
\limsup_{n\rightarrow\infty}\frac{f_{n}(x)}{g_n(x)}&=&\lim_{j\rightarrow\infty}\frac{f_{n_j}(x)}{g_{n_j}(x)}
\\&\leq&\lim_{j\rightarrow\infty}\frac{S_{n_j}f_\xi(x)+n_j\xi}{S_{n_j}g_\xi(x)-n_j\xi}
\\
&=&\lim_{j\rightarrow\infty} \frac{\int f_\xi\,
\mathrm{d}\delta_{x,n_j}+\xi}{\int g_\xi\,
\mathrm{d}\delta_{x,n_j}-\xi}\\
&=&\frac{\int f_\xi\,\mathrm{d}\mu+\xi}{\int
g_\xi\,\mathrm{d}\mu-\xi}.
\end{eqnarray*}
Letting $\xi\rightarrow 0$, by (\ref{approx}) we have
\[
\limsup_{n\rightarrow\infty}\frac{f_{n}(x)}{g_{n}(x)}\leq\frac{
\mathcal{F}_*(\mu)}{\mathcal{G}_*(\mu)}.
\]
This implies that
\[\sup_{x\in
E_{\Phi,\Psi}(\alpha)}\limsup_{n\rightarrow\infty}\frac{f_n(x)}{g_n(x)}\le\sup
\Big\{ \frac{\mathcal{F}_*(\mu)}{\mathcal{G}_*(\mu)}:\ \mu\in
\mathcal{M}_T\ \hbox{and}\
\frac{\Phi_*(\mu)}{\Psi_*(\mu)}=\alpha\Big\}.\]

To prove the reverse inequality, note that the set of
$T$-invariant measure $\mu$ such that
$\frac{\Phi_*(\mu)}{\Psi_*(\mu)}=\alpha$ is a compact subset of
$\mathcal{M}_T$.  Choose such a measure $\mu$ so that
$$\frac{\mathcal{F}_*(\mu)}{\mathcal{G}_*(\mu)}=\sup
\Big\{ \frac{\mathcal{F}_*(\mu)}{\mathcal{G}_*(\mu)}:\ \mu\in
\mathcal{M}_T\ \hbox{and}\
\frac{\Phi_*(\mu)}{\Psi_*(\mu)}=\alpha\Big\}.$$ Fix a small number
$\xi>0$, and $\varphi_\xi$ and $\psi_\xi$ are continuous functions
given by (\ref{Def-aap}) approximating $\Phi$ and $\Psi$
respectively.  By the specification property of the TDS $(X, T)$,
there exists $x_0\in G_\mu$, i.e., $\delta_{x_0,n}\rightarrow\mu$
as $n\rightarrow\infty$. Hence,
\[
\lim_{n\rightarrow\infty}\frac 1 n S_n
\varphi_\xi(x_0)=\lim_{n\rightarrow\infty} \int \varphi_\xi\,
\mathrm{d}\delta_{x_0,n} =\int  \varphi_\xi\, \mathrm{d} \mu
\]
and
\[
\lim_{n\rightarrow\infty}\frac 1 n S_n
\psi_\xi(x_0)=\lim_{n\rightarrow\infty} \int \psi_\xi\,
\mathrm{d}\delta_{x_0,n} =\int  \psi_\xi\, \mathrm{d} \mu.
\]
Therefore, for all sufficiently large $n$ we have that
\begin{eqnarray*}
\frac{\varphi_n(x_0)}{\psi_n(x_0)}\le
\frac{S_n\varphi_\xi(x_0)+\xi}{S_n\psi_\xi(x_0)-\xi}\le \frac{\int
\varphi_\xi\,\mathrm{d}\mu+2\xi}{\int
\psi_\xi\,\mathrm{d}\mu-2\xi}\le
\frac{\Phi_*(\mu)+3\xi}{\Psi_*(\mu)-3\xi}.
\end{eqnarray*}
Similarly, for all sufficiently large $n$ we can prove that
\[
\frac{\varphi_n(x_0)}{\psi_n(x_0)}\ge\frac{\Phi_*(\mu)-3\xi}{\Psi_*(\mu)+3\xi}.
\]
Hence,
\[
\lim_{n\rightarrow\infty}\frac{\varphi_n(x_0)}{\psi_n(x_0)}=\frac{\Phi_*(\mu)}{\Psi_*(\mu)}=\alpha.
\]
This means that $x_0\in E_{\Phi,\Psi}(\alpha)$. On the other hand,
by (\ref{ds34}) we have that
\begin{eqnarray*}
\limsup_{n\rightarrow\infty}\frac{f_n(x_0)}{g_n(x_0)}&\geq&
\limsup_{n\rightarrow\infty}\frac{S_n f_\xi(x_0)-n\xi}{S_n g_\xi(x_0)+n\xi}\\
&=&\limsup_{n\rightarrow\infty}\frac{\int f_\xi\,
\mathrm{d}\delta_{x_0,n}-\xi}{\int g_\xi\,
\mathrm{d}\delta_{x_0,n}+\xi}\\
&=&\frac{\int f_\xi\, \mathrm{d}\mu-\xi}{\int g_\xi\,
\mathrm{d}\mu+\xi}.
\end{eqnarray*}
Letting $\xi\rightarrow 0$, by (\ref{approx}) we have
\[
\limsup_{n\rightarrow\infty}\frac{f_n(x_0)}{g_n(x_0)}\geq\frac{\mathcal{F}_*(\mu)}{\mathcal{G}_*(\mu)}.
\]
It follows that $$\sup_{x\in
E_{\Phi,\Psi}(\alpha)}\limsup_{n\rightarrow\infty}\frac{f_n(x)}{g_n(x)}\ge\sup
\Big\{ \frac{\mathcal{F}_*(\mu)}{\mathcal{G}_*(\mu)}:\ \mu\in
\mathcal{M}_T\ \hbox{and}\
\frac{\Phi_*(\mu)}{\Psi_*(\mu)}=\alpha\Big\}.$$
This finishes the proof of Theorem \ref{thm.e}.
\end{proof}

Next we turn to prove Theorem \ref{thm.f}, some of the methods was
used by Thompson in \cite{th} for the multifractal analysis of
Birkhoff averages of a continuous function, in that paper he
proved that the irregular set is either empty or carries full
topological pressure when the system has specification property.
His ideas is originally due to the work of Takens and Verbitskiy
\cite{tv1,tv2} and
 Chen {\it et al.} \cite{chen}.
\begin{proof}[Proof of Theorem \ref{thm.f}] We first prove that
\begin{eqnarray}\label{thm.f.1}
\sup\Big \{\frac{\mathcal{F}_{*}(\mu)}{\mathcal{G}_{*}(\mu)}:\
\mu\in \mathcal{E}_T\Big\}=\sup
\Big\{\frac{\mathcal{F}_{*}(\mu)}{\mathcal{G}_{*}(\mu)}:\ \mu\in
\mathcal{M}_T\Big\}.
\end{eqnarray}
Since $\mathcal{E}_T\subset \mathcal{M}_T$, it is clear that
\begin{eqnarray*}
\sup\Big \{\frac{\mathcal{F}_{*}(\mu)}{\mathcal{G}_{*}(\mu)}:\
\mu\in \mathcal{E}_T\Big\}\le\sup
\Big\{\frac{\mathcal{F}_{*}(\mu)}{\mathcal{G}_{*}(\mu)}:\ \mu\in
\mathcal{M}_T\Big\}.
\end{eqnarray*}
To prove the reverse inequality, since the map $\mu\mapsto \frac{\mathcal{F}_{*}(\mu)}{\mathcal{G}_{*}(\mu)}$ is continuous on the space $\mathcal{M}_T$ we can choose a $T$-invariant measure
$\mu\in \mathcal{M}_T$ such that
\[
\frac{\mathcal{F}_{*}(\mu)}{\mathcal{G}_{*}(\mu)}=\sup
\Big\{\frac{\mathcal{F}_{*}(\mu)}{\mathcal{G}_{*}(\mu)}:\ \mu\in
\mathcal{M}_T\Big\}.
\]
By (3) of Theorem \ref{fhthm}, using the ergodic decomposition
theorem (see \cite{k10}) we have that
$$\mathcal{F}_*(\mu)=\int_{\mathcal{E}_T}\mathcal{F}_*(m)\,\mathrm{d}\tau(m)~~\text{and}~~
\mathcal{G}_*(\mu)=\int_{\mathcal{E}_T}\mathcal{G}_*(m)\,\mathrm{d}\tau(m),$$
where $\tau$ is a probability measure on the space $\mathcal{M}_T$
such that $\tau(\mathcal{E}_T)=1$. Let
$\lambda:=\frac{\mathcal{F}_*(\mu)}{\mathcal{G}_*(\mu)}$, we have
\begin{eqnarray*}
0&=&\int_{\mathcal{E}_T}\mathcal{F}_*(m)-\lambda\mathcal{G}_*(m)\,\mathrm{d}\tau(m)\\
&=&\int_{\mathcal{E}_T}\mathcal{G}_*(m)\Big(\frac{\mathcal{F}_*(m)}{\mathcal{G}_*(m)}-\lambda\Big)\,\mathrm{d}\tau(m).
\end{eqnarray*}
Since $\mathcal{G}_*(m)\ge \sigma>0$ and
$\frac{\mathcal{F}_*(m)}{\mathcal{G}_*(m)}\le\lambda$ for each
$m\in \mathcal{E}_T$, for $\tau-$a.e. $m\in \mathcal{E}_T$ we have
\[
\lambda= \frac{\mathcal{F}_*(m)}{\mathcal{G}_*(m)}.
\]
Thus formula (\ref{thm.f.1}) immediately follows.

Next, we prove the second equality of the theorem. It is clear
that
$$\sup \{\frac{\mathcal{F}_{*}(\mu)}{\mathcal{G}_{*}(\mu)}:\ \mu\in
\mathcal{M}_T\}\ge\sup
\{\frac{\mathcal{F}_{*}(\mu)}{\mathcal{G}_{*}(\mu)}:\ \mu\in
\bigcup_{x\in\widehat{X}_{\Phi,\Psi}}\mathcal{V}(x)\},$$ since
every measure $\mu\in \bigcup_{x\in
\widehat{X}_{\Phi,\Psi}}\mathcal{V}(x)$ is $T$-invariant.

To prove the reverse inequality,  by formula (\ref{thm.f.1}) it
suffices to prove that $\mathcal{E}_T\subset \bigcup_{x\in
\widehat{X}_{\Phi,\Psi}}\mathcal{V}(x)$.

Given $\mu_1\in \mathcal{E}_T$, there must exists some
$T$-invariant ergodic measure $\mu_2$ satisfying
$$\frac{\Phi_*(\mu_1)}{\Psi_*(\mu_1)}\neq\frac{\Phi_*(\mu_2)}{\Psi_*(\mu_2)},$$ since
$\inf\limits_{\mu\in\mathcal{M}_T}\frac{\Phi_*(\mu)}{\Psi_*(\mu)}<\sup\limits_{\mu\in\mathcal{M}_T}\frac{\Phi_*(\mu)}{\Psi_*(\mu)}$.
For $i
= 1, 2$, choose a point $x_i$ satisfies $$
\lim_{n\rightarrow\infty}\frac{\varphi_n
(x_i)}{\psi_n(x_i)}=
\frac{\Phi_*(\mu_i)}{\Psi_*(\mu_i)}.$$  Let $m_k := m(\epsilon/2^k)$ be as in the definition of
specification and $N_k$ be a sequence of integers chosen to grow
to infinity sufficiently rapidly that $N_{k+1} > \exp{
\sum_{i=1}^{k}(N_i + m_i)}$. We define a sequence of points
$\{z_i\}_{i\ge 1}\subset X$ inductively using the specification
property. For $x,y\in X$ and $n\in \mathbb{N}$, define a new metric on $X$ as follows
$$d_n(x, y) := \max\bigr\{d(T^ix, T^iy) : i = 0,1,\cdots,
n-1\bigr\}.$$ And the dynamical ball centered at $x$ of radius $r$
and length $n$ is denoted by $B_n(x,r):=\{y\in X:\ d_n(x,y)<r\}$.
Let $t_1 = N_1$,  $t_k = t_{k-1}+m_k+N_k$ for $k\geq 2$ and
$\rho(k):=(k+1)(\mathrm{mod}\ 2)+1$. Let $z_1= x_1$. Let $z_2$
satisfy
\[
d_{N_1}(z_2, z_1) < \frac \epsilon 4\  \hbox{and}\
d_{N_2}(T^{N_1+m_2}z_2 , x_2) <  \frac \epsilon 4.
\]
For $k>2$, let $z_k$ satisfy
\[
d_{t_{k-1}}(z_k,z_{k-1})< \frac{\epsilon}{2^k}\  \hbox{and}\
d_{N_k}(T^{t_{k-1}+m_k}z_k , x_{\rho(k)})<\frac{\epsilon}{2^k}.
\]
Note that if $q\in \overline{B}_{t_k}(z_k, \epsilon / 2^{k-1})$,
then
\begin{eqnarray*}
d_{t_{k-1}}(q,z_{k-1})&\leq&
d_{t_{k-1}}(q,z_k)+d_{t_{k-1}}(z_k,z_{k-1})\\
&\le&\frac{\epsilon}{2^{k-1}}+\frac{\epsilon}{2^k}<\frac{\epsilon}{2^{k-2}},
\end{eqnarray*}
and thus $\overline{B}_{t_k}(z_k, \epsilon /
2^{k-1})\subset\overline{B}_{t_{k-1}}(z_{k-1}, \epsilon /
2^{k-2})$. Hence, we can define a point by
\[
p:=\bigcap_{k\geq 1}\overline{B}_{t_k}(z_k, \epsilon / 2^{k-1}).
\]
For each continuous function $\phi\in C(X)$, since the orbit of
$p$ alternates between approximating increasingly long orbit
segments of $x_1$ and $x_2$, we can show that
\begin{eqnarray}\label{thm.f.2}
\frac{1}{t_k}S_{t_k}\phi(p)\rightarrow \int \phi\,
\mathrm{d}\mu_{\rho(k)}~(k\rightarrow\infty).\end{eqnarray} Thus
$\delta_{p,t_{2k+1}}\rightarrow \mu_1$,
$\delta_{p,t_{2k}}\rightarrow \mu_2$ as $k\rightarrow\infty$ and,
so $\mu_i\in \mathcal{V}(p)$ for $i=1,2$. To complete the proof of the second equality of 
this theorem, it is left to show that
$$p\in \widehat{X}_{\Phi,\Psi}.$$
Indeed, fix a small number $\xi>0$, by the definition of
asymptotically additive potentials there exist continuous
functions $\varphi_\xi$ and $\psi_\xi$ approximating $\Phi$ and
$\Psi$ respectively. Applying (\ref{thm.f.2}) for continuous
functions $\varphi_\xi$ and $\psi_\xi$, for all sufficiently large
$k$ we have
\begin{eqnarray*}
\frac{\varphi_{t_{2k}}(p)}{\psi_{t_{2k}}(p)}&\le&
\frac{S_{t_{2k}}\varphi_\xi(p)+t_{2k}\xi}{S_{t_{2k}}\psi_\xi(p)-t_{2k}\xi}=\frac{\int
\varphi_\xi\,\mathrm{d} \delta_{p,t_{2k}}+\xi}{\int
\psi_\xi\,\mathrm{d} \delta_{p,t_{2k}}-\xi}\\
&\le&\frac{\int\varphi_\xi\,\mathrm{d} \mu_2+2\xi
}{\int\psi_\xi\,\mathrm{d} \mu_2-2\xi}\le
\frac{\Phi_*(\mu_2)+3\xi}{\Psi_*(\mu_2)-3\xi}.
\end{eqnarray*}
Similarly, we can prove that
\[
\frac{\varphi_{t_{2k}}(p)}{\psi_{t_{2k}}(p)}\ge
\frac{\Phi_*(\mu_2)-3\xi}{\Psi_*(\mu_2)+3\xi}.
\]
Hence,
\[
\lim_{k\rightarrow\infty}\frac{\varphi_{t_{2k}}(p)}{\psi_{t_{2k}}(p)}=\frac{\Phi_*(\mu_2)}{\Psi_*(\mu_2)}.
\]
By the same arguments, we can prove that
\[
\lim_{k\rightarrow\infty}\frac{\varphi_{t_{2k+1}}(p)}{\psi_{t_{2k+1}}(p)}=\frac{\Phi_*(\mu_1)}{\Psi_*(\mu_1)}.
\]
This implies that $p\in \widehat{X}_{\Phi,\Psi}$. Hence,
$$\mathcal{E}_T\subset \bigcup_{x\in \widehat{X}_{\Phi,\Psi}}\mathcal{V}(x).$$
This completes the proof of the second equality of Theorem \ref{thm.f}.

To prove the last equality, by the Proposition \ref{PthmF} below and \eqref{thm.f.1} we have
$$
\sup_{x\in \widehat{X}_{\Phi,\Psi}}\limsup_{n\to\infty}\frac{f_n(x)}{g_n(x)}\le \sup
\Big\{\frac{\mathcal{F}_{*}(\mu)}{\mathcal{G}_{*}(\mu)}:\ \mu\in
\mathcal{E}_T\Big\}.
$$
Next we prove the reverse inequality. Pick a $T$-invariant ergodic measure $\widetilde{\mu}_1$  such that
$$\frac{\mathcal{F}_*(\widetilde{\mu}_1)}{\mathcal{G}_*(\widetilde{\mu}_1)}=\sup
\Big\{\frac{\mathcal{F}_{*}(\mu)}{\mathcal{G}_{*}(\mu)}:\ \mu\in
\mathcal{E}_T\Big\}.$$
Then we choose another $T$-invariant ergodic measure $\widetilde{\mu}_2$ so that
$$\frac{\Phi_*(\widetilde{\mu}_1)}{\Psi_*(\widetilde{\mu}_1)}\neq\frac{\Phi_*(\widetilde{\mu}_2)}{\Psi_*(\widetilde{\mu}_2)}.$$
Repeating the arguments in the proof of the second equality, we can find a point $\widetilde{p}\in \widehat{X}_{\Phi,\Psi}$
such that $\delta_{\widetilde{p},t_{2k+1}}\rightarrow \widetilde{\mu}_1$ and
$\delta_{\widetilde{p},t_{2k}}\rightarrow \widetilde{\mu}_2$ as $k\rightarrow\infty$.  Fix a small number $\xi>0$, by the definition of
asymptotically additive potentials there exist continuous
functions $f_\xi$ and $g_\xi$ approximating $\mathcal{F}$ and
$\mathcal{G}$ respectively. Applying \eqref{thm.f.2} for continuous
functions $f_\xi$ and $g_\xi$, for all sufficiently large
$k$ we have
\begin{eqnarray*}
\frac{f_{t_{2k+1}}(\widetilde{p})}{g_{t_{2k+1}}(\widetilde{p})}&\le&
\frac{S_{t_{2k+1}}f_\xi(\widetilde{p})+t_{2k+1}\xi}{S_{t_{2k+1}}g_\xi(\widetilde{p})-t_{2k+1}\xi}=\frac{\int
f_\xi\,\mathrm{d} \delta_{\widetilde{p},t_{2k+1}}+\xi}{\int
g_\xi\,\mathrm{d} \delta_{\widetilde{p},t_{2k+1}}-\xi}\\
&\le&\frac{\int f_\xi\,\mathrm{d} \widetilde{\mu}_1+2\xi
}{\int g_\xi\,\mathrm{d} \widetilde{\mu}_1-2\xi}\le
\frac{\mathcal{F}_*(\widetilde{\mu}_1)+3\xi}{\mathcal{G}_*(\widetilde{\mu}_1)-3\xi}.
\end{eqnarray*}
Similarly, we can prove that
\[
\frac{f_{t_{2k+1}}(\widetilde{p})}{g_{t_{2k+1}}(\widetilde{p})}\ge
\frac{\mathcal{F}_*(\widetilde{\mu}_1)-3\xi}{\mathcal{G}_*(\widetilde{\mu}_1)+3\xi}.
\]
Hence,
\[
\lim_{k\rightarrow\infty}\frac{f_{t_{2k+1}}(\widetilde{p})}{g_{t_{2k+1}}(\widetilde{p})}=\frac{\mathcal{F}_*(\widetilde{\mu}_1)}{\mathcal{G}_*(\widetilde{\mu}_1)}.
\]
This implies that
$$
\sup_{x\in \widehat{X}_{\Phi,\Psi}}\limsup_{n\to\infty}\frac{f_n(x)}{g_n(x)}\ge \sup
\Big\{\frac{\mathcal{F}_{*}(\mu)}{\mathcal{G}_{*}(\mu)}:\ \mu\in
\mathcal{E}_T\Big\}.
$$
The proof of this theorem is completed.
\end{proof}

\begin{proposition}\label{PthmF} Let $\mathcal{F}=\{f_n\}_{n\ge 1}$ and $\mathcal{G}=\{g_n\}_{n\ge 1}$ be two AAPs. Assume that $\mathcal{G}$ satisfies \eqref{denom}, then $$\sup_{x\in X}\limsup_{n\to\infty}\frac{f_n(x)}{g_n(x)}=\sup_{x\in \mathcal{R}}\lim_{n\to\infty}\frac{f_n(x)}{g_n(x)}=
\limsup_{n\to\infty}\max_{x\in X}\frac{f_n(x)}{g_n(x)}=
\sup\{\frac{\mathcal{F}_*(\mu)}{\mathcal{G}_*(\mu)}
\mid \mu\in \mathcal{M}_T\}$$
where $\mathcal{R}=\{x\in X: \lim\limits_{n\to\infty}\frac{f_n(x)}{g_n(x)}~\text{exists}\}$.
\end{proposition}
\begin{proof}It is clear that
$$\sup_{x\in \mathcal{R}}\lim_{n\to\infty}\frac{f_n(x)}{g_n(x)}\le \sup_{x\in X}\limsup_{n\to\infty}\frac{f_n(x)}{g_n(x)}\le \limsup_{n\to\infty}\max_{x\in X}\frac{f_n(x)}{g_n(x)}.$$
Given a $T$-invariant ergodic measure $\mu$, by Theorem \ref{fhthm} there exist a point $x$ such that
$$\lim_{n\to\infty}\frac{f_n(x)}{g_n(x)}=\frac{\mathcal{F}_*(\mu)}{\mathcal{G_*(\mu)}}.$$
This together with \eqref{thm.f.1} imply that
$$\sup_{x\in \mathcal{R}}\lim_{n\to\infty}\frac{f_n(x)}{g_n(x)}\ge\sup\{\frac{\mathcal{F}_*(\mu)}{\mathcal{G}_*(\mu)}
\mid \mu\in \mathcal{M}_T\}.$$

On the other hand, for each $n$ choose a point $x_n\in X$ such that
$$\frac{f_n(x_n)}{g_n(x_n)}=\max_{x\in X}\frac{f_n(x)}{g_n(x)}.$$
Put $\mu_n=\frac1n \sum_{i=0}^{n-1}\delta_{T^ix_n}$, and let $\mu$ be a weak$^*$ limit point of $\{\mu_n\}_{n\ge 1}$. Without loss of generality, assume that
$\mu_n\to \mu$ as $n\to\infty$.  Fix a small number $\xi>0$, let $f_\xi$ and $g_\xi$ be the continuous
functions in the definition of
asymptotically additive potentials that  approximating $\mathcal{F}$ and
$\mathcal{G}$ respectively. Hence, for all sufficiently large $n$ we have
\begin{eqnarray*}
\frac{f_n(x_n)}{g_n(x_n)}\le \frac{\frac1n S_nf_\xi(x_n)+\xi}{\frac1n S_ng_\xi(x_n)-\xi}=\frac{\int f_\xi\mathrm{d}\mu_n+\xi}{\int g_\xi\mathrm{d}\mu_n-\xi},
\end{eqnarray*}
and this yields that
$$\limsup_{n\to\infty}\max_{x\in X}\frac{f_n(x)}{g_n(x)}\le\lim_{n\to\infty}\frac{\int f_\xi\mathrm{d}\mu_n+\xi}{\int g_\xi\mathrm{d}\mu_n-\xi}=\frac{\int f_\xi\mathrm{d}\mu+\xi}{\int g_\xi\mathrm{d}\mu-\xi}.$$
Letting  $\xi\to 0$ yields that
$$\limsup_{n\to\infty}\max_{x\in X}\frac{f_n(x)}{g_n(x)}\le\frac{\mathcal{F}_*(\mu)}{\mathcal{G}_*(\mu)}.
$$
This yields the desired result.
\end{proof}

\section{Application to suspension flows} This section provides
applications of our main results to suspension flows. Let $T :
X\rightarrow X$ be a homeomorphism of a compact metric space $X$
with metric $d$, and $\tau: X \rightarrow (0,\infty)$ is a
continuous \emph{roof function}. The \emph{suspension space} is
defined as follows: \[ X_\tau := \{(x, s)\in X\times\mathbb{R}  :
0\leq s \leq\tau (x)\},
\]
where $(x, \tau (x))$ is identified with $(T(x), 0)$ for all $x$.
There is a natural topology on $X_\tau$ which makes $X_\tau$ a
compact topological space. This topology is induced by a distance
introduced by Bowen and Walters in \cite{bw} (see the appendix in
\cite{bs} for details).

The \emph{suspension flow} $\Theta=(\theta_t)_t$ on $X_\tau$ is
defined by $\theta_t(x, s) = (x, s + t)$. For a continuous
function $\Phi:X_\tau\rightarrow \mathbb{R}$, we associate the
function $\varphi:X\rightarrow \mathbb{R}$ by
\begin{eqnarray} \label{ass-func}\varphi(x) = \int_{0}^{\tau(x)}
\Phi(x, t)\,\mathrm{d}t.\end{eqnarray} Since the roof function
$\tau$ is continuous, so is the function $\varphi$. If $x\in X$
and $s\in [0, \tau(x)]$, we have (see \cite[Lemma 5.3]{th})
\begin{eqnarray}\label{flow}
\liminf_{T\rightarrow\infty}\frac 1 T \int_{0}^{T}
\Phi(\theta_t(x, s))\,\mathrm{d}t &=& \liminf_{
n\rightarrow\infty}\frac{ S_n\varphi(x)}{ S_n\tau(x)}.
\end{eqnarray}
The above formula remains true if we replace $\liminf$ by
$\limsup$.

 For any $T$-invariant measure $\mu\in
\mathcal{M}_T$, we define the measure $\mu_\tau$ on the suspension
space $X_\tau$ by
\[
\int_{X_\tau} \Phi \,\mathrm{d}\mu_\tau
=\frac{\int_X\varphi\,\mathrm{d}\mu}{\int_X\tau\,\mathrm{d}\mu}
\]
for all continuous function $\Phi\in C(X_\tau)$, where $\varphi$
is defined as (\ref{ass-func}). It is well-known that the measure
$\mu_\tau$ is $\Theta-$invariant, i.e.,
$\mu(\theta_{t}^{-1}A)=\mu(A)$ for all $t\geq 0$ and measurable
sets $A$. The map $\mathcal{L}: \mathcal{M}_T \rightarrow
\mathcal{M}_\Theta$ given by $\mu\mapsto\mu_\tau$ is a bijection,
where $\mathcal{M}_\Theta$ is the space of all $\Theta-$invariant
measures on the suspension space $X_\tau$.

For any real number $\alpha$, we consider the level set associated
to a continuous $\Phi:X_\tau\rightarrow \mathbb{R}$ as follows:
\begin{eqnarray*}K(\Phi,\alpha) := \Big\{(x, s)\in X_\tau :
\lim_{T\rightarrow\infty}\frac 1 T \int_{0}^{T} \Phi(\theta_t(x,
s))\,\mathrm{d}t = \alpha\Big\}.
\end{eqnarray*}

\begin{theorem}\label{flow-erg}Let  $T:X\rightarrow
X$  be a homeomorphism on a compact metric space $(X,d)$ with the
specification property, $\tau: X \rightarrow (0,\infty)$ a
continuous roof function, and $(X_\tau, \Theta)$ the corresponding
suspension flow over $X$. Let $\Phi, H\in C( X_\tau)$, for each
$\alpha$ with $K(\Phi,\alpha)\neq \emptyset$ we have
\begin{eqnarray*}
\begin{aligned}
\sup_{(x,s)\in K(\Phi,\alpha)}\limsup_{T\rightarrow\infty} \frac 1
T\int_{0}^{T} H(\theta_t(x,s))\,\mathrm{d}t=\sup \Big\{ \int
H\,\mathrm{d}\mu_\tau:\mu_\tau\in \mathcal{M}_\Psi,\ \int \Phi\,
\mathrm{d}\mu_\tau=\alpha\Big\}\\
=\sup \Big\{ \int H\,\mathrm{d}\mu_\tau:
\mu_\tau\in\bigcup_{(x,s)\in
K(\Phi,\alpha)}\mathcal{V}((x,s))\Big\}
\end{aligned}
\end{eqnarray*}
where $\mathcal{V}((x,s))$ is the set of limits point of the
sequence
$\delta_{(x,s),T}:=\frac{1}{T}\int_{0}^{T}\delta_{\theta_t(x,s)}\,\mathrm{d}t$.
\end{theorem}
\begin{proof}Let $\varphi, h\in C(X)$ be the continuous function on $X$ associated to $\Phi$ and $H$ respectively,
that is
\[
\varphi(x)= \int_{0}^{\tau(x)} \Phi(x,
t)\,\mathrm{d}t~~\text{and}~~ h(x)=\int_{0}^{\tau(x)} H(x,
t)\,\mathrm{d}t.
\]
For a real number $\alpha$ with $K(\Phi,\alpha)\neq \emptyset$,
let
$$\Gamma(\alpha):=\Big\{x \in X: \lim_{ n\rightarrow\infty}\frac{
S_n\varphi(x)}{ S_n\tau(x)}=\alpha \Big\}.$$ Using (\ref{flow}),
$x\in \Gamma(\alpha)$ if and only if $(x,s)\in K(\Phi,\alpha)$ for all $0\le s\le \tau(x)$.
Hence,
\[
\sup_{(x,s)\in K(\Phi,\alpha)}\limsup_{T\rightarrow\infty} \frac 1
T\int_{0}^{T} H(\theta_t(x,s))\,\mathrm{d}t=\sup_{x\in
\Gamma(\alpha)}\limsup_{ n\rightarrow\infty}\frac{ S_nh(x)}{
S_n\tau(x)}.
\]
In Theorem \ref{thm.e}, consider particular asymptotically
additive potentials $g_n(x)=\psi_n(x)=S_n\tau(x)$,
$f_n(x)=S_nh(x)$ and $\varphi_n(x)=S_n\varphi(x)$ for each $x\in
X$ and $n\in \mathbb{N}$. Then we have
\begin{eqnarray*}
\sup_{x\in \Gamma(\alpha)}\limsup_{ n\rightarrow\infty}\frac{
S_nh(x)}{ S_n\tau(x)}&=&\sup\Big\{ \frac{\int h\,\mathrm{d}\mu
}{\int \tau\,\mathrm{d}\mu}:\mu\in
\mathcal{M}_T~\text{and}~\frac{\int
\varphi\,\mathrm{d}\mu }{\int \tau\,\mathrm{d}\mu}=\alpha\Big\}\\
&=&\sup \Big\{ \int H\,\mathrm{d}\mu_\tau:\mu_\tau\in
\mathcal{M}_\Psi\ \hbox{and}\ \int \Phi\,
\mathrm{d}\mu_\tau=\alpha\Big\},
\end{eqnarray*}
the above last equality holds since
\[
\int H\,\mathrm{d}\mu_\tau= \frac{\int h\,\mathrm{d}\mu }{\int
\tau\,\mathrm{d}\mu},~\int \Phi\, \mathrm{d}\mu_\tau=\frac{\int
\varphi\,\mathrm{d}\mu }{\int \tau\,\mathrm{d}\mu}
\]
for some $\mu\in \mathcal{M}_T$ and the map $\mu\mapsto\mu_\tau$
is a bijection. To finish the proof of the theorem, by Theorem
\ref{thm.e} it suffices to prove that
\begin{eqnarray}\label{last}
\sup \Big\{ \int H\,\mathrm{d}\mu_\tau:
\mu_\tau\in\bigcup_{(x,s)\in
K(\Phi,\alpha)}\mathcal{V}((x,s))\Big\}=\sup\Big\{ \frac{\int
h\,\mathrm{d}\mu }{\int \tau\,\mathrm{d}\mu}:\mu\in\bigcup_{x\in
\Gamma(\alpha)} \mathcal{V}(x)\Big\}.
\end{eqnarray}
Indeed, for each $\mu\in \mathcal{V}(x)$ for some $x\in
\Gamma(\alpha)$, we know that $(x,s)\in K(\Phi,\alpha)$ for any
$0\le s<\tau(x)$. Furthermore, by (\ref{flow}) we know that the
corresponding measure $\mu_\tau$ is a limit point of the sequence
$\delta_{(x,s),T}$. Conversely, for each
$\mu_\tau\in\bigcup\limits_{(x,s)\in
K(\Phi,\alpha)}\mathcal{V}((x,s))$, by a standard argument we can
show that the corresponding measure $\mu$ is a limit point of the
sequence $\delta_{x,n}$. This observation yields (\ref{last}).
\end{proof}

For a continuous function $\Phi:X_\tau\rightarrow \mathbb{R}$ on
the suspension space, we consider its irregular set as follows:
\[
\widehat{X}_{\Phi}:=\Big\{(x,s)\in
X_{\tau}:\lim_{T\rightarrow\infty}\frac 1 T \int_{0}^{T}
\Phi(\theta_t(x, s))\mathrm{d}t~\text{does not exist} \Big\}.
\]

\begin{theorem}
Let  $T:X\rightarrow X$  be a homeomorphism on a compact metric
space $(X,d)$ with the specification property, $\tau: X
\rightarrow (0,\infty)$ a continuous roof function, and $(X_\tau,
\Theta)$ the corresponding suspension flow over $X$. Let $\Phi, H\in
C( X_\tau)$, if $$\inf_{\mu_\tau\in \mathcal{M}_\Theta}\int \Phi
\,\mathrm{d}\mu_\tau<\sup_{\mu_\tau\in \mathcal{M}_\Theta}\int
\Phi \,\mathrm{d}\mu_\tau$$ then we have
\begin{eqnarray*}
\sup \Big\{\int H\,\mathrm{d}\mu_\tau:\ \mu_\tau\in
\mathcal{M}_\Theta\Big\}&=&\sup \Big\{\int H\,\mathrm{d}\mu_\tau:\
\mu_\tau\in
\bigcup_{(x,s)\in\widehat{X}_{\Phi}}\mathcal{V}((x,s))\Big\}\\
&=&\sup_{(x,s)\in\widehat{X}_{\Phi}}\limsup_{T\to\infty} \frac 1
T\int_{0}^{T} H(\theta_t(x,s))\,\mathrm{d}t.
\end{eqnarray*}
\end{theorem}
\begin{proof}First note that
\[
\sup \Big\{\int H\,\mathrm{d}\mu_\tau:\ \mu_\tau\in
\mathcal{M}_\Theta\Big\}=\sup \Big\{\frac{\int
h\,\mathrm{d}\mu}{\int \tau\,\mathrm{d}\mu}:\mu\in
\mathcal{M}_T\Big\}.
\]
Let $\widehat{X}_{\varphi,\tau}=\Big\{x\in
X:\lim\limits_{n\rightarrow\infty}\frac{S_n\varphi(x)}{S_n\tau(x)}~\text{does
not exist} \Big\}$, by (\ref{flow}) we know that
$x\in \widehat{X}_{\varphi,\tau}$ if and only if $(x,s)\in \widehat{X}_{\Phi}$ for all $0\le s\le \tau(x)$. Using similar
arguments as the proof of Theorem \ref{flow-erg}, we have that
\[\sup \Big\{\int H\,\mathrm{d}\mu_\tau:\ \mu_\tau\in
\bigcup_{(x,s)\in\widehat{X}_{\Phi}}\mathcal{V}((x,s))\Big\}=\sup
\Big\{\frac{\int h\,\mathrm{d}\mu}{\int
\tau\,\mathrm{d}\mu}:\mu\in\bigcup_{x\in
\widehat{X}_{\varphi,\tau}} \mathcal{V}(x)\Big\}.\]
On the other hand, note that
\[
\sup_{(x,s)\in \widehat{X}_\Phi}\limsup_{T\to\infty}\frac 1T \int_{0}^{T} H(\theta_t(x,s))\,\mathrm{d}t
=\sup_{x\in\widehat{X}_{\varphi,\tau}}\lim_{n\to\infty} \frac{S_nh(x)}{S_n\tau(x)}
\]
and  $\inf\limits_{\mu_\tau\in \mathcal{M}_\Theta}\int \Phi
\,\mathrm{d}\mu_\tau<\sup\limits_{\mu_\tau\in \mathcal{M}_\Theta}\int
\Phi \,\mathrm{d}\mu_\tau$ is equivalent to $\inf\limits_{\mu\in \mathcal{M}_T}\frac{\int \varphi
\,\mathrm{d}\mu}{\int \tau\,\mathrm{d}\mu}<\sup\limits_{\mu\in \mathcal{M}_T}\frac{\int
\varphi \,\mathrm{d}\mu}{\int
\tau \,\mathrm{d}\mu}$.
Consider $\mathcal{G}=\Psi=\{S_n\tau\}_{n\ge 1}$, $\Phi=\{S_n\varphi\}_{n\ge 1}$ and $\mathcal{F}=\{S_nh\}_{n\ge 1}$ in Theorem
\ref{thm.f}, the desired result immediately follows.
\end{proof}
 \noindent {\bf Acknowledgements.}  This work was finally completed when I am supported by CSC to visit Pennsylvania State University in the academic year 2013-2014. I would like to take this chance to thank Professor Yakov Pesin for the warm hospitality.
This work is partially supported by NSFC (11371271).

\end{document}